\documentclass[reqno,oneside,12pt]{amsart}

\usepackage{hyperref}
\usepackage{longtable}
\usepackage[ansinew]{inputenc}
\usepackage{graphicx}
\usepackage{color}
\usepackage[numbers, square]{natbib}
\usepackage{mathrsfs}
\usepackage{bbm}
\usepackage{tikz}
\usepackage{mathptmx}
\usepackage{amsmath,amsthm,amsfonts,amssymb}
\usepackage{dsfont,extarrows,enumerate}
\usepackage{fullpage}
\usepackage{xcolor}

\numberwithin{equation}{section}

\theoremstyle{plain} \newtheorem{theorem}{Theorem}[section]
\theoremstyle{plain} \newtheorem{proposition}[theorem]{Proposition}
\theoremstyle{plain} \newtheorem{lemma}[theorem]{Lemma}
\theoremstyle{plain} \newtheorem{corollary}[theorem]{Corollary}
\theoremstyle{definition} 
\theoremstyle{definition} 
\theoremstyle{remark} \newtheorem{remark}[theorem]{Remark}
\theoremstyle{remark} \newtheorem{example}[theorem]{Example}

\newcommand{\1}{\mathbbm{1}}

\newcommand{\GCD}{{\rm GCD}\,}
\newcommand{\LCM}{{\rm LCM}\,}
\newcommand{\Vol}{{\bf Vol}\,}

\newcommand{\E}{\mathbb{E}}
\newcommand{\R}{\mathbb R}
\newcommand{\N}{\mathbb N}

\newcommand{\dodn}{\overset{{\rm d}}{\underset{n\to\infty}\longrightarrow}}

\renewcommand{\epsilon}{\varepsilon}

\renewcommand{\geq}{\geqslant}
\renewcommand{\leq}{\leqslant}

\usepackage{fancybox}
\newlength{\querylen}
\begin{document}
\title{Asymptotics of arithmetic functions of GCD and LCM of random integers in hyperbolic regions}

\author{Alexander Iksanov}
\address{Alexander Iksanov, Faculty of Computer Science and Cybernetics, Taras Shev\-chen\-ko National University of Kyiv, 01601 Kyiv, Ukraine}
\email{iksan@univ.kiev.ua}

\author{Alexander Marynych}
\address{Alexander Marynych, Faculty of Computer Science and Cybernetics, Taras Shev\-chen\-ko National University of Kyiv, 01601 Kyiv, Ukraine}
\email{marynych@unicyb.kiev.ua}

\author{Kilian Raschel}
\address{Kilian Raschel, Universit\'e d'Angers, CNRS, Laboratoire Angevin de Recherche en Math\'e-matiques, 49000 Angers, France}\email{raschel@math.cnrs.fr}
\thanks{This project has received funding from the European Research Council (ERC) under the European Union's Horizon 2020 research and innovation programme under the Grant Agreement No.\ 759702.}

\date{\today}

\begin{abstract}
We prove limit theorems for the greatest common divisor and the least common multiple of random integers. While the case of integers uniformly distributed on a hypercube with growing size is classical, we look at the uniform distribution on sublevel sets of multivariate symmetric polynomials, which we call hyperbolic regions. Along the way of deriving our main results, we obtain some asymptotic estimates for the number of integer points in these hyperbolic domains, when their size goes to infinity.
\end{abstract}


\maketitle

\section{Introduction}

Let $f:\N\to \mathbb{C}$ be an arithmetic function, with $\mathbb N$ denoting $\{1,2,\ldots\}$. The motivation for the present paper comes from the recent study of hyperbolic sums
\begin{equation}\label{eq:f_dim_two_def}
   f_{G}(n):=\sum_{ij\leq n}f(\GCD(i,j))\quad\text{and}\quad f_{L}(n):=\sum_{ij\leq n}f(\LCM(i,j))
\end{equation}
carried out in \cite{Heyman+Toth:2021}, where the authors derived asymptotics of $f_{G}(n)$ and $f_{L}(n)$, as $n\to\infty$, for certain classes of arithmetic functions $f$. For example, Theorem~2.2 in \cite{Heyman+Toth:2021} yields the following asymptotics
\begin{equation}\label{eq:HT2021_example}
\lim_{n\to\infty}\frac{f_{G}(n)}{n\log n}=\frac{1}{\zeta(2)}\sum_{k=1}^{\infty}\frac{f(k)}{k^2},
\end{equation}
provided that $f(n)=o(n^{\beta}\log^{\delta}n)$, as $n\to\infty$, for some $\beta<1$, $\delta\in\mathbb{R}$ and with $\zeta$ being the Riemann zeta-function.

To set up the scene,
recast \eqref{eq:f_dim_two_def} and \eqref{eq:HT2021_example} in the probabilistic language as follows. Assume that on a certain probability space $(\Omega,\mathcal{F},\mathbb{P})$, there is a sequence of random vectors $\bigl((V_1^{(n)},V_2^{(n)})\bigr)_{n\in\N}$ such that, for every fixed $n$, $(V_1^{(n)},V_2^{(n)})$ has a
uniform distribution on the {\it finite} set
\begin{equation*}
   H_{2,2}(n):=\{(i_1,i_2)\in\N^2: i_1 i_2\leq n\}
\end{equation*}
(the choice of notation $H_{2,2}$ will be explained below, see \eqref{eq:def_hyperbolic_regions}). This means that, 
for all $(i_1,i_2)\in H_{2,2}(n)$,
\begin{equation*}
   \mathbb{P}\{(V_1^{(n)},V_2^{(n)})=(i_1,i_2)\}=\frac{1}{\vert H_{2,2}(n)\vert}.
\end{equation*}
Then
\begin{equation}
\label{eq:f_def_introduction}
   f_{G}(n)=\vert H_{2,2}(n)\vert \cdot\E f(\GCD(V_1^{(n)},V_2^{(n)}))\quad\text{and}\quad f_{L}(n)=\vert H_{2,2}(n)\vert \cdot\E f(\LCM(V_1^{(n)},V_2^{(n)})).
\end{equation}
Taking into account the asymptotics
\begin{equation*}
   \vert H_{2,2}(n)\vert =\sum_{i_1=1}^{n}\sum_{i_2=1}^{\lfloor n/i_1 \rfloor } 1=\sum_{i_1=1}^{n}\lfloor n/i_1 \rfloor ~\sim~ n\log n,\quad n\to\infty,
\end{equation*}
where $\lfloor x\rfloor$ denotes the integer part of $x\in\mathbb{R}$ and the notation $a_n\sim b_n$ means that $\lim_{n\to\infty}(a_n/b_n)=1$, we conclude that \eqref{eq:HT2021_example} is equivalent to
\begin{equation}\label{eq:HT2021_example_distr}
\lim_{n\to\infty}\E f(\GCD(V_1^{(n)},V_2^{(n)}))=\frac{1}{\zeta(2)}\sum_{k=1}^{\infty}\frac{f(k)}{k^2}.
\end{equation}
Remarkably, the quantity on the right-hand side coincides with $\E f(U^{(2,\infty)})$, where by Theorem~1 in \cite{Diaconis+Erdos:2004}, the distribution of $U^{(2,\infty)}$ is the distributional limit of
$\GCD(Z_1^{(n)},Z_2^{(n)})$ as $n\to\infty$, the pair $(Z_1^{(n)},Z_2^{(n)})$ being  uniformly distributed in the square $\{1,2,\ldots,n\}^2$. Since \eqref{eq:HT2021_example_distr} holds for all bounded arithmetic functions, it actually tells us that there is the convergence in distribution
\begin{equation*}
   \lim_{n\to\infty}\mathbb{P}\{\GCD(V_1^{(n)},V_2^{(n)})=m\}=\mathbb{P}\{U^{(2,\infty)}=m\},\quad m\in\mathbb{N}.
\end{equation*}
Therefore, $\GCD(V_1^{(n)},V_2^{(n)})$ for large $n$ behaves as 
the GCD of two independent integers picked uniformly at random from $\{1,2,\ldots,n\}$.

We shall show in the present paper that it is not a coincidence but
rather a simple instance of a much deeper and general phenomenon. This observation will allow us to extend some results in \cite{Heyman+Toth:2021} to an arbitrary dimension and cover more general hyperbolic regions defined by the standard symmetric polynomials.

\subsection*{Acknowledgments}
We thank the anonymous referee for several useful comments and suggestions.

\section{Hyperbolic regions and hyperbolic sums}

Fix $r\in\N$ and $1\leq\ell\leq r$, and let $P_\ell(x_1,x_2,\ldots,x_r)$ be the $\ell$-th standard symmetric polynomial in $r$ variables, that is,
\begin{equation*}
   P_\ell(x_1,x_2,\ldots,x_r):=\sum_{1\leq i_1<i_2<\cdots<i_\ell\leq r} x_{i_1}x_{i_2}\cdots x_{i_\ell}.
\end{equation*}
In particular,
\begin{equation*}
   P_1(x_1,x_2,\ldots,x_r)=x_1+x_2+\cdots+x_r\quad\text{and}\quad P_r(x_1,x_2,\ldots,x_r)=x_1x_2\cdots  x_r.
\end{equation*}
Now we introduce `discrete' hyperbolic regions in $\N^r$ given, for $n\geq \binom{r}{\ell}$, by
\begin{equation}\label{eq:def_hyperbolic_regions}
   H_{\ell,r}(n):=\{(i_1,\ldots,i_r)\in\N^r: P_\ell(i_1,i_2,\ldots,i_r)\leq n\}.
\end{equation}
Observe that the condition $n\geq \binom{r}{\ell}$ ensures $H_{\ell,r}(n)\neq \varnothing$. Moreover, for $r=\ell=2$, \eqref{eq:def_hyperbolic_regions} is consistent with the definition of $H_{2,2}(n)$ in the introduction. In what follows, we fix $r\in\{2,3,\ldots\}$ and $\ell\in\{1,2,\ldots,r\}$. Let $(V_1^{(n)},V_2^{(n)},\ldots,V_r^{(n)})$ be a random vector uniformly distributed in $H_{\ell,r}(n)$, that is,
\begin{equation*}
   \mathbb{P}\{(V_1^{(n)},V_2^{(n)},\ldots,V_r^{(n)})=(i_1,i_2,\ldots,i_r)\}=\frac{1}{\vert H_{\ell,r}(n)\vert },\quad (i_1,i_2,\ldots,i_r)\in H_{\ell,r}(n),\quad n\geq \binom{r}{\ell}.
\end{equation*}
We shall also use the following `continuous' counterparts of the discrete regions $H_{\ell,r}(n)$:
\begin{equation}
\label{eq:def_continuous_hyperbolic}
   \mathcal{H}_{\ell,r}(c):=\{(x_1,x_2,\ldots,x_r)\in \mathbb{R}_{\geq 0}^r: P_\ell(x_1,x_2,\ldots,x_r)\leq c\},\quad 1\leq \ell\leq r,\quad c>0,
\end{equation}
where $\mathbb{R}_{\geq 0}:=[0,\infty)$. See Figure~\ref{fig:hyperbolic_regions} for a few illustrations. Note that $\mathcal{H}_{\ell,r}(c)=c^{1/\ell}\mathcal{H}_{\ell,r}(1)$, by the homogeneity property of $P_\ell$. Let $\Vol$ denote the $r$-dimensional Lebesgue measure on $\mathbb{R}^r$. It is clear that $\Vol(\mathcal{H}_{1,r}(1))=1/r!<\infty$ and $\Vol(\mathcal{H}_{r,r}(1))=\infty$.
It will be shown in
Lemma~\ref{lem:finite_volumes} below that the volumes of all intermediate regions are finite. Since these volumes will play an important
role in what follows, 
we introduce the following notation:
\begin{equation*}
   \mathcal{V}_{\ell,r}:=\Vol(\mathcal{H}_{\ell,r}(1))=\underbrace{\int_0^\infty\cdots\int_0^\infty}_{r\text{ times}}\1_{\{P_{\ell}(y_1,y_2,\ldots,y_r)\leq 1\}}{\rm d}y_1\cdots{\rm d}y_r,\quad 1\leq \ell<r.
\end{equation*}
We do not know whether $\mathcal{V}_{\ell,r}$ admits a closed-form expression, for $1<\ell<r$.

\begin{figure}
\begin{center}
\includegraphics[width=4cm]{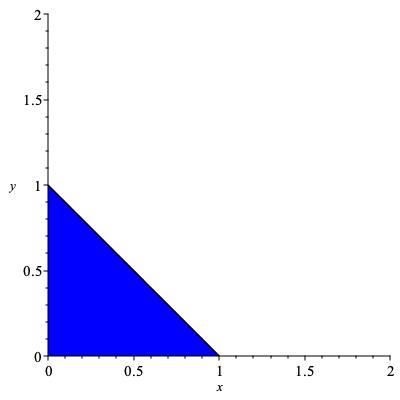}
\includegraphics[width=4cm]{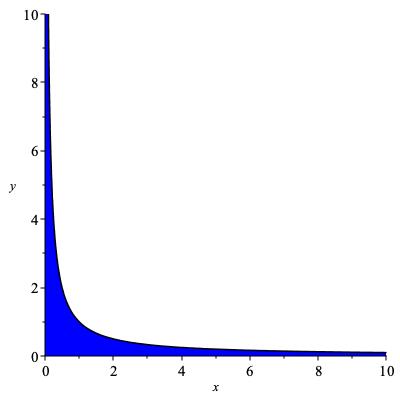}\\
\includegraphics[width=4cm]{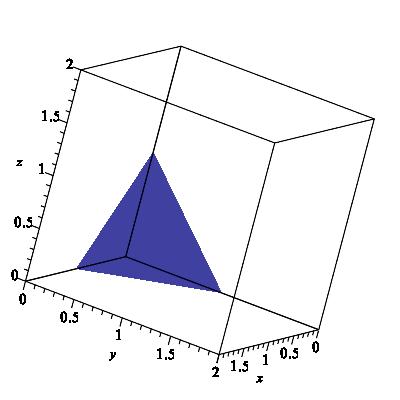}
\includegraphics[width=4cm]{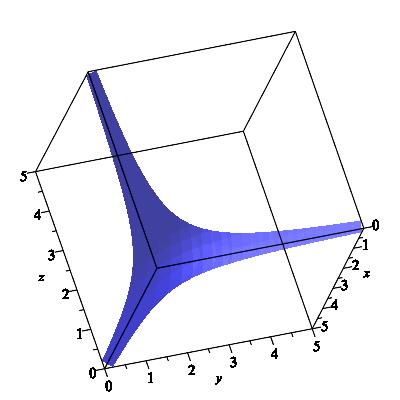}
\includegraphics[width=4cm]{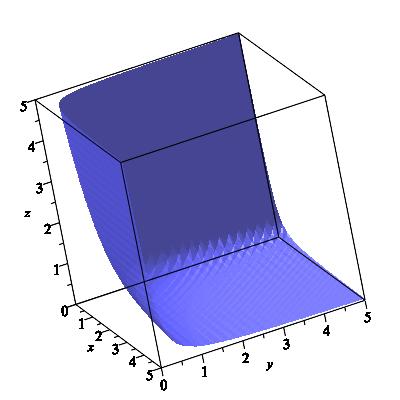}
\end{center}
\caption{Hyperbolic regions defined by \eqref{eq:def_continuous_hyperbolic} with $c=1$. The first row: $\mathcal H_{1,2}(1)$ and $\mathcal H_{2,2}(1)$. The second row: $\mathcal H_{1,3}(1)$, $\mathcal H_{2,3}(1)$ and $\mathcal H_{3,3}(1)$. The adjective `hyperbolic' stems from the fact that, for $r\geq 2$ and $1<\ell\leq r$, the set $\{(x_1,x_2)\in \mathbb{R}^2_{\geq 0}:P_{\ell}(x_1,x_2,c_3,\ldots,c_r)=c\}$ defines either a hyperbola or an empty set for all $c_3,\ldots,c_r > 0$ and $c>0$. This term is not quite appropriate in the case $\ell=1$, in which $\mathcal{H}_{1,r}(c)$ is an $r$-dimensional polytope.}
\label{fig:hyperbolic_regions}
\end{figure}

For an arithmetic function $f:\mathbb{N}\to \mathbb{C}$ and $n\geq \binom{r}{\ell}$, consider the random variables
\begin{equation}\label{eq:f_def}
f_{\ell,r,G}(n):=f(\GCD(V_1^{(n)},V_2^{(n)},\ldots,V_r^{(n)}))~\text{and}~f_{\ell,r,L}(n):=f(\LCM(V_1^{(n)},V_2^{(n)},\ldots,V_r^{(n)})).
\end{equation}
The following equalities extend  formula \eqref{eq:f_def_introduction}:
\begin{align*}
\E f_{\ell,r,G}(n)&=\frac{1}{\vert H_{\ell,r}(n)\vert }\sum_{(i_1,\ldots,i_r)\in H_{\ell,r}(n)}f(\GCD(i_1,i_2,\ldots,i_r)),\\
\E f_{\ell,r,L}(n)&=\frac{1}{\vert H_{\ell,r}(n)\vert }\sum_{(i_1,\ldots,i_r)\in H_{\ell,r}(n)}f(\LCM(i_1,i_2,\ldots,i_r)).
\end{align*}
Thus, deriving the asymptotics of the hyperbolic sums $\sum_{(i_1,\ldots,i_r)\in H_{\ell,r}(n)}f(\GCD(i_1,i_2,\ldots,i_r))$ and $\sum_{(i_1,\ldots,i_r)\in H_{\ell,r}(n)}f(\LCM(i_1,i_2,\ldots,i_r))$ is equivalent to finding the asymptotics of the counting function $\vert H_{\ell,r}(n)\vert $ and 
the expectations $\E f_{\ell,r,G}(n)$ and $\E f_{\ell,r,L}(n)$, respectively. The latter will be obtained
for various functions $f$ from the corresponding distributional limit theorems for
\begin{equation*}
   \GCD(V_1^{(n)},V_2^{(n)},\ldots,V_r^{(n)})\quad \text{and}\quad \LCM(V_1^{(n)},V_2^{(n)},\ldots,V_r^{(n)}).
\end{equation*}

\section{Statement of the main results}
\subsection{First properties of the uniform distribution on \texorpdfstring{$H_{\ell,r}(n)$}{Hrl(n)}}

We start with some basic asymptotic properties of the distribution of $(V_1^{(n)},V_2^{(n)},\ldots,V_r^{(n)})$, which, we recall, is the uniform distribution on the set $H_{\ell,r}(n)$ defined in \eqref{eq:def_hyperbolic_regions}.

\begin{proposition}\label{prop:joint_limit_l<r}
Assume that $r\geq 2$ and $1\leq \ell<r$. Then, for
$\alpha_1,\ldots,\alpha_r>0$,
\begin{equation*}
   \lim_{n\to\infty}\mathbb{P}\left\{V_1^{(n)}\leq \alpha_1 n^{1/\ell},\ldots,V_r^{(n)}\leq \alpha_r n^{1/\ell}\right\}=\frac{1}{\mathcal{V}_{\ell,r}}\int_0^{\alpha_1}\cdots\int_0^{\alpha_r}\1_{\{P_{\ell}(y_1,y_2,\ldots,y_r)\leq 1\}}{\rm d}y_1\cdots{\rm d}y_r.
\end{equation*}
\end{proposition}
Proposition~\ref{prop:joint_limit_l<r}, as well as all subsequent 
results stated in this section, will be proved in Section~\ref{sec:proofs}.

In the case $r=\ell$, the limit relation is of different nature, for 
the volume $\mathcal{V}_{r,r}$ is infinite. In the sequel, we find it more convenient to write distributional limit relations using `$\dodn$' notation. 
Specifically, for fixed $r\in\mathbb{N}$, the notation 
\begin{equation*}
   (X_{n,1},\ldots,X_{n,r}) \dodn (X_{1},\ldots,X_{r})
\end{equation*} 
 means that $\mathbb{P}\{X_{n,1}\leq x_1,\ldots,X_{n,r}\leq x_r\}\to \mathbb{P}\{X_{1}\leq x_1,\ldots,X_{r}\leq x_r\}$, as $n\to\infty$, for each continuity point $(x_1,\ldots,x_r)$
of the distribution function $(y_1,\ldots,y_r)\mapsto \mathbb{P}\{X_1\leq y_1,\ldots,X_r\leq y_r\}$.

Let $Z_1,\ldots,Z_{r-1}$ be independent random variables with continuous uniform distribution on $[0,1]$. Denote by $Z^{(1)}<Z^{(2)}<\ldots<Z^{(r-1)}$ their order statistics. Put $Z^{(r)}:=1$,
\begin{equation*}
   J_1:=Z^{(1)}\quad\text{and}\quad J_k=Z^{(k)}-Z^{(k-1)},\quad k=2,\ldots,r.
\end{equation*}
\begin{proposition}\label{prop:joint_limit_l=r}
Assume that $r=\ell\geq 2$. Then
\begin{equation}
\label{eq:joint_limit_l=r}
   \left(\frac{\log V_1^{(n)}}{\log n},\frac{\log V_2^{(n)}}{\log n},\ldots,\frac{\log V_r^{(n)}}{\log n}\right)\dodn \left(J_1,J_2,\ldots,J_r\right),
\end{equation}
or, equivalently,
\begin{equation}
\label{eq:joint_limit_l=r_alternative}
   \left(\frac{\log V_1^{(n)}}{\log n},\frac{\log (V_1^{(n)}V_2^{(n)})}{\log n},\ldots,\frac{\log (V_1^{(n)}\cdots V_r^{(n)})}{\log n}\right)\dodn \left(Z^{(1)},Z^{(2)},\ldots,Z^{(r)}\right).
\end{equation}
\end{proposition}

The next result deals with limit theorems for the product $V_{1}^{(n)}V_{2}^{(n)}\cdots V_{r}^{(n)}$.
\begin{proposition}\label{prop:product}
Assume that $r=\ell\geq 2$. Then
\begin{equation}\label{eq:product_conv1}
   \frac{V_{1}^{(n)}V_{2}^{(n)}\cdots V_{r}^{(n)}}{n}\dodn U_{r,r},
\end{equation}
where $U_{r,r}$ has a continuous uniform distribution on $[0,1]$.

Assume that $1\leq \ell<r$. Then
\begin{equation}\label{eq:product_conv2}
\frac{V_{1}^{(n)}V_{2}^{(n)}\cdots  V_{r}
^{(n)}}{n^{r/\ell}}\dodn U_{\ell,r},
\end{equation}
where $U_{\ell,r}$ has the distribution function
\begin{equation}\label{eq:product_distribution}
\mathbb{P}\{U_{\ell,r}\leq x\}=\frac{1}{\mathcal{V}_{\ell,r}}\int_0^\infty\cdots\int_0^\infty\1_{\{P_{\ell}(y_1,y_2,\ldots,y_r)\leq 1,\;y_1y_2\cdots y_r\leq x\}}{\rm d}y_1\cdots{\rm d}y_r,\quad x\in \left[0,\,x^{\ast}_{\ell,r}\right],
\end{equation}
and $x^{\ast}_{\ell,r}:=\binom{r}{\ell}^{-r/\ell}$.

In both cases, for
$n\geq \binom{r}{\ell}$ (with $\ell=r$ in the first case),
\begin{equation}\label{eq:product_bounded}
\mathbb{P}\left\{0\leq \frac{V_{1}^{(n)}V_{2}^{(n)}\cdots V_{r}
^{(n)}}{n^{r/\ell}}\leq 1\right\}=1.
\end{equation}
In particular, all power moments of positive orders in relations \eqref{eq:product_conv1} and \eqref{eq:product_conv2} converge to the corresponding moments of the limit random variables.
\end{proposition}

\begin{example}
The distribution function of $U_{1,2}$ can be explicitly calculated and takes the following form:
\begin{equation*}
   \mathbb{P}\{U_{1,2}\leq x\}=1-\sqrt{1-4x}+2x\log\left(\frac{1+\sqrt{1-4x}}{1-\sqrt{1-4x}}\right),\quad x\in [0,\,1/4],
\end{equation*}
with a density (the derivative)
$x\mapsto 2\log\left(\frac{1+\sqrt{1-4x}}{1-\sqrt{1-4x}}\right)\1_{[0,\,1/4]}(x)$. For other values of $\ell<r$, there seems to be no simple closed form expression for $\mathbb{P}\{U_{\ell,r}\leq x\}$.
\end{example}

\subsection{Arithmetic properties of the uniform distribution on \texorpdfstring{$H_{\ell,r}(n)$}{Hrl(n)}}
Our next result shows that without {\it any} assumptions on the function $f$, the
random variables $f_{\ell,r,G}(n)$ in \eqref{eq:f_def} converge in distribution, as $n\to\infty$.

As a preparation, we introduce a collection 
of random variables, which is of major importance for the subsequent analysis. Let $\mathcal{P}$ denote the set of prime numbers and
$(\mathcal{G}_k(p))_{k\in\mathbb{N},
p\in\mathcal{P}}$ be a collection
of mutually independent random variables with the following geometric distributions
\begin{equation*}
   \mathbb{P}\{\mathcal{G}_k(p)=j\}=\left(1-\frac{1}{p}\right)\frac{1}{p^j},\quad j=0,1,2,\ldots
\end{equation*}
Finally, let $\lambda_p(n)\in\{0,1,2,\ldots\}$ denote the multiplicity of a prime $p$ in the prime decomposition of an integer $n$, that is, 
\begin{equation*}
   n=\prod_{p\in\mathcal{P}}p^{\lambda_p(n)}.
\end{equation*}
\begin{theorem}\label{thm:main1}
Let $f:\N\to \mathbb{C}$ be an arithmetic function. Then
\begin{equation}
\label{thm:main1_claim}
   f_{\ell,r,G}(n)\dodn f\left(\prod_{p\in\mathcal{P}}p^{\min_{k=1,\ldots,r}\mathcal{G}_k(p)}\right).
\end{equation}
\end{theorem}
\begin{remark}
The distribution of the random
variable 
\begin{equation*}
U^{(r,\infty)}:=\prod_{p\in\mathcal{P}}p^{\min_{k=1,\ldots,r}\mathcal{G}_k(p)}
\end{equation*}
can be characterized as follows. Since the minimum of independent geometric variables has again a geometric distribution with the parameter being the product of the parameters of individual variables, the Mellin transform of $U^{(r,\infty)}$ is given by
\begin{align*}
\E \bigl((U^{(r,\infty)})^{s}\bigr)=\prod_{p\in\mathcal{P}}\E p^{s\min_{k=1,\ldots,r}\mathcal{G}_k(p)}=\prod_{p\in\mathcal{P}}\left(\sum_{j=0}^{\infty}p^{sj}\left(1-\frac{1}{p^r}\right)\frac{1}{p^{rj}}\right)=\frac{\zeta(r-s)}{\zeta(r)},\quad s<r-1.
\end{align*}
We have used Euler's product formula for the last equality.
\end{remark}

Theorem~\ref{thm:main2} below is a limit theorem for the $\LCM$.
\begin{theorem}\label{thm:main2}
The following convergence in distribution holds true:
\begin{align}\label{eq:lcm_conv1}
\frac{\LCM(V_1^{(n)},V_2^{(n)},\ldots,V_r^{(n)})}{V_1^{(n)}V_2^{(n)}\cdots V_r^{(n)}} &\dodn \prod_{p\in\mathcal{P}}
p^{\max_{k=1,\ldots,r}\mathcal{G}_k(p)-\sum_{k=1}^{r}\mathcal{G}_k(p)},\medskip
\\ \label{eq:lcm_conv2}
\frac{\LCM(V_1^{(n)},V_2^{(n)},\ldots,V_r^{(n)})}{n^{r/\ell}} &\dodn U_{\ell,r}\prod_{p\in\mathcal{P}}
p^{\max_{k=1,\ldots,r}\mathcal{G}_k(p)-\sum_{k=1}^{r}\mathcal{G}_k(p)},
\end{align}
where the random variable $U_{\ell,r}$ on the right-hand side of \eqref{eq:lcm_conv2} is independent of the $\mathcal{G}_k(p)$ and has the distribution given by \eqref{eq:product_distribution} if $\ell<r$, and has the uniform distribution on $[0,1]$ if $\ell=r$. Moreover, in both relations \eqref{eq:lcm_conv1} and \eqref{eq:lcm_conv2}, all power moments of positive orders converge to the corresponding moments of the limit random variables.
\end{theorem}

Our last result is concerned with the asymptotic behavior of the average $\E f_{\ell,r,L}(n)$. Recall that a real-valued measurable function $f$ defined in a neighbourhood of $+\infty$ is called regularly varying at $\infty$ if there exists $\beta\in\mathbb{R}$ such that, for all $\lambda>0$,
\begin{equation*}
   \lim_{t\to\infty}\frac{f(\lambda t)}{f(t)}=\lambda^{\beta}.
\end{equation*} 
The parameter $\beta$ is called the index of regular variation of $f$ at $\infty$. We refer to \cite{BGT} for a comprehensive information on regularly varying functions.

\begin{corollary}\label{corr:LCM_moments}
Let $f:\R^r_{\geq 0}\to\R$ be a locally bounded function which varies regularly at $\infty$ of index
$\beta>0$. Then,
as $n\to\infty$,
\begin{multline*}
   \E f_{\ell,r,L}(n)=\frac{1}{\vert H_{\ell,r}(n)\vert }\sum_{(i_1,\ldots,i_r)\in H_{\ell,r}(n)}f(\LCM(i_1,i_2,\ldots,i_r))\\
\sim~\E (U_{\ell,r}^{\beta})\E \left(\biggl(\prod_{p\in\mathcal{P}}
p^{\max_{k=1,\ldots,r}\mathcal{G}_k(p)-\sum_{k=1}^{r}\mathcal{G}_k(p)}\biggr)^{\beta}\right)f(n^{r/\ell}).
\end{multline*}
\end{corollary}

\section{Proofs of the main results}
\label{sec:proofs}

We start with the detailed analysis of the counting functions $\vert H_{\ell,r}(n)\vert $, which is an
essential ingredient for the proofs of our main results.

\subsection{Properties of the counting function when \texorpdfstring{$\ell=1$}{l=1} and \texorpdfstring{$\ell=r$}{l=r}.}

We first consider the case $\ell=1$ and $r\in\N$. Then
\begin{equation*}
   H_{1,r}(n)=\{(i_1,\ldots,i_r)\in\N^r: i_1+i_2+\cdots+i_r\leq n\},
\end{equation*}
and there is the obvious exact formula $\vert H_{1,r}(n)\vert =\binom{n}{r}$,
which entails that, as $n\to\infty$,
\begin{equation}\label{corr:H1r(n)}
   \vert H_{1,r}(n)\vert ~\sim~\frac{n^r}{r!}.
\end{equation}
Assume now that 
$r=\ell\geq 2$. Then
\begin{equation*}
   H_{r,r}(n)=\{(i_1,\ldots,i_r)\in\N^r: i_1 i_2\cdots i_r\leq n\}.
\end{equation*}
Although there is no simple exact formula for the cardinality of $H_{r,r}(n)$, one can easily derive the exact growth rate of $\vert H_{r,r}(n)\vert $. This is given in the next proposition.
\begin{proposition}\label{prop:Hrr(n)}
For 
fixed $r\geq 2$, as $n\to\infty$, 
\begin{equation*}
   \vert H_{r,r}(n)\vert =\frac{n\log^{r-1}n}{(r-1)!}+O(n\log^{r-2}n).
\end{equation*}
\end{proposition}
\begin{proof}
Put $W_r(n):=\vert H_{r,r}(n)\vert$. Then 
$W_1(n)=n$ and
\begin{equation}\label{eq:Hrr(n)_proof1}
W_r(n)=\sum_{i=1}^{n}W_{r-1}\left(\left\lfloor \frac{n}{i}\right\rfloor \right),\quad r\geq 2.
\end{equation}
The claim of Proposition~\ref{prop:Hrr(n)} follows by induction on $r$ with the help of 
the asymptotic relation
\begin{equation*}
   \sum_{i=1}^{n}\left\lfloor\frac{n}{i}\right\rfloor\log^{k-1}\left(\left\lfloor \frac{n}{i}\right\rfloor \right)=\sum_{i=1}^{n}\frac{n}{i}\log^{k-1}\left(\frac{n}{i}\right)+O(n\log^{k-1} n)=\frac{n\log^k n}{k}+O(n\log^{k-1} n),\quad n\to\infty,
\end{equation*}
which holds for every fixed $k\in\N$.
\end{proof}

\begin{corollary}\label{corr:Hrr(n)}
For 
fixed $r\in\N$, the sequence $(\vert H_{r,r}(n)\vert )_{n\in\N}$ is regularly varying at $\infty$ of index 
$1$, that is, for each
$\lambda>0$,
\begin{equation*}
   \lim_{n\to\infty}\frac{\vert H_{r,r}(\lfloor \lambda n\rfloor)\vert }{\vert H_{r,r}(n)\vert }=\lambda.
\end{equation*}
\end{corollary}

The result of Corollary~\ref{corr:Hrr(n)} is less precise than that of Proposition~\ref{prop:Hrr(n)}. It is stated here only for comparison to its counterpart, Corollary~\ref{cor:counting_function_reg_var}, which treats the case $1<\ell<r$.

\subsection{Properties of the counting function when \texorpdfstring{$1<\ell<r$}{1<l<r}.}

Comparing \eqref{corr:H1r(n)} and Proposition~\ref{prop:Hrr(n)} and keeping in mind the homogeneity properties of $P_\ell$, one could think that the asymptotics of $\vert H_{\ell,r}(n)\vert $ in the intermediate regimes should be of the form $C_r n^{r/\ell}\log^{\ell-1}n$. This, however, turns out to be wrong in that 
there is no logarithmic factor, that is, the correct answer is $\vert H_{\ell,r}(n)\vert \sim C_r n^{r/\ell}$ for 
an appropriate $C_r>0$. This is, in fact, a consequence of the finiteness of the volumes $\mathcal{V}_{\ell,r}$ for $\ell<r$.

\begin{lemma}\label{lem:finite_volumes}
For all $r\geq 2$ and $1\leq \ell<r$, $\Vol(\mathcal{H}_{\ell,r}(c))=c^{r/\ell}\Vol(\mathcal{H}_{\ell,r}(1))<\infty$.
\end{lemma}
\begin{proof}
We proceed in two steps. First, we show that
\begin{equation}\label{eq:finite_volume_r-1_r}
   \Vol(\mathcal{H}_{r-1,r}(1))<\infty.
\end{equation}
As a second step, we prove that
\begin{equation}\label{eq:finite_volume_l_r_inclusion}
\mathcal{H}_{\ell,r}(1)\subseteq \mathcal{H}_{r-1,r}(r),\quad \ell<r.
\end{equation}
To check \eqref{eq:finite_volume_r-1_r}, observe that
\begin{equation*}
   \Vol(\mathcal{H}_{r-1,r}(1))= \int_0^\infty\cdots\int_0^\infty\1_{\{P_{r-1}(y_1,y_2,\ldots,y_r)\leq 1\}}{\rm d}y_1\cdots{\rm d}y_r.
\end{equation*}
Changing the variables 
$z_j:=(y_1y_2\cdots y_r)/y_j$ or, equivalently,
$y_j=(z_1 z_2\cdots z_r)^{1/(r-1)}z_j^{-1}$,  $j=1,\ldots,r$, we conclude that 
the partial derivatives are given by 
\begin{equation*}
   \frac{\partial y_j}{\partial z_k}=
\begin{cases}
(r-1)^{-1}(z_1z_2\cdots z_r)^{1/(r-1)}z_j^{-1}z_k^{-1}, &j\neq k,\\
\frac{2-r}{r-1}(z_1z_2\cdots z_r)^{1/(r-1)}z_j^{-2}, & j=k.\\
\end{cases}
\end{equation*}
Thus, the Jacobian determinant $J$ is equal to
\begin{multline*}
   J=
(z_1z_2\cdots z_r)^{\frac{r}{r-1}}\left\vert
\begin{matrix}
\frac{2-r}{r-1}\frac{1}{z_1^2} & \frac{1}{r-1}\frac{1}{z_1 z_2} & \cdots & \frac{1}{r-1}\frac{1}{z_1 z_r}\\
\frac{1}{r-1}\frac{1}{z_2 z_1} & \frac{2-r}{r-1}\frac{1}{z_2^2} & \cdots & \frac{1}{r-1}\frac{1}{z_2 z_r}\\
\vdots & \vdots & \ddots & \vdots\\
\frac{1}{r-1}\frac{1}{z_r z_1} & \frac{1}{r-1}\frac{1}{z_r z_2} & \cdots & \frac{2-r}{r-1}\frac{1}{z_r^2}
\end{matrix}
\right\vert =
\frac{1}{(z_1z_2\cdots z_r)^{\frac{r-2}{r-1}}}\left\vert
\begin{matrix}
\frac{2-r}{r-1} & \frac{1}{r-1} & \cdots & \frac{1}{r-1}\\
\frac{1}{r-1}& \frac{2-r}{r-1} & \cdots & \frac{1}{r-1}\\
\vdots & \vdots & \ddots & \vdots\\
\frac{1}{r-1} & \frac{1}{r-1} & \cdots & \frac{2-r}{r-1}
\end{matrix}
\right\vert \\ =\frac{(-1)^{r-1}}{(r-1)(z_1z_2\cdots z_r)^{\frac{r-2}{r-1}}},
\end{multline*}
whence 
\begin{multline*}
\Vol(\mathcal{H}_{r-1,r}(1))=\frac{1}{r-1}\int_0^\infty\cdots\int_0^\infty\frac{\1_{\{z_1+z_2+\cdots+z_r \leq 1\}}}{(z_1z_2\cdots z_r)^{\frac{r-2}{r-1}}}{\rm d}z_1\cdots{\rm d}z_r\\
\leq \frac{1}{r-1}\int_0^1\cdots\int_0^1\frac{1}{(z_1z_2\cdots z_r)^{\frac{r-2}{r-1}}}{\rm d}z_1\cdots{\rm d}z_r=(r-1)^{r-1}<\infty.
\end{multline*}
This proves 
\eqref{eq:finite_volume_r-1_r}.

Turning to
\eqref{eq:finite_volume_l_r_inclusion}, pick $(x_1,x_2,\ldots ,x_r)\in\mathcal{H}_{\ell,r}(1)$. Then
\begin{equation*}
   x_{i_1}x_{i_2}\cdots x_{i_\ell}\leq 1,\quad\text{ for every } \ell\text{-tuple}\quad 1\leq i_1< i_2<\cdots <i_\ell\leq r.
\end{equation*}
Fix $k=1,\ldots,r$ and multiply the above inequalities over all $\ell$-tuples taken from $\{1,2,\ldots,k-1,k+1,\ldots,r\}$. This yields $x_{1}x_2\cdots x_{k-1}x_{k+1}\cdots x_r\leq 1$ and thereupon $P_{r-1}(x_1,x_2,\ldots,x_r)\leq r$, meaning that $(x_1,x_2,\ldots,x_r)\in \mathcal{H}_{r-1,r}(r)$.
\end{proof}
\begin{proposition}\label{prop:Hrl(n)}
For 
fixed $r\geq 2$ and $\ell<r$, 
\begin{equation*}
   \lim_{n\to\infty} \frac{\vert H_{\ell,r}(n)\vert}{n^{r/\ell}}=\mathcal{V}_{\ell,r}.
\end{equation*}
\end{proposition}
\begin{proof}
By homogeneity of $P_\ell$,
\begin{equation*}
   \frac{\vert H_{\ell,r}(n)\vert}{n^{r/\ell}} =\frac{1}{n^{r/\ell}}\sum_{i_1=1}^{\infty}\cdots\sum_{i_r=1}^{\infty}\1_{\{P_\ell(i_1,\ldots,i_r)\leq n\}}=\frac{1}{(n^{1/\ell})^r}\sum_{i_1=1}^{\infty}\cdots\sum_{i_r=1}^{\infty}\1_{\{P_\ell(i_1/n^{1/\ell},\ldots,i_r/n^{1/\ell})\leq 1\}}.
\end{equation*}
The claim follows from Proposition~\ref{prop:dri} (see Appendix~\ref{sec:app}) applied to the function $g(y_1,\ldots,y_r):=\1_{\{P_\ell(y_1,\ldots,y_r)\leq 1\}}$. Indeed, while this function is obviously coordinatewise nonincreasing, 
its integrability follows from Lemma~\ref{lem:finite_volumes}.
\end{proof}

\begin{corollary}\label{cor:counting_function_reg_var}
For 
fixed $r\geq 2$ and $1\leq \ell<r$, the sequence $(\vert H_{\ell,r}(n)\vert)_{n\geq \binom{r}{\ell}}$ 
is regularly varying at $\infty$ of index 
$r/\ell$, that is, for each
$\lambda>0$,
\begin{equation*}
   \lim_{n\to\infty}\frac{\vert H_{\ell,r}(\lfloor n\lambda\rfloor)\vert }{\vert H_{\ell,r}(n)\vert }=\lambda^{r/\ell}.
\end{equation*}
\end{corollary}

\begin{proposition}\label{prop:counting_function_asymp}
For
fixed $r\geq 2$, $1\leq \ell\leq r$ and $t_1,\ldots,t_r>0$, 
\begin{equation*}
   \lim_{n\to\infty}\frac{\vert \{(i_1,\ldots,i_r)\in\N^r:P_\ell(t_1 i_1,\ldots,t_r i_r)\leq n\}\vert }{\vert H_{\ell,r}(n)\vert }=\left(\prod_{k=1}^{r}t_k\right)^{-1}.
\end{equation*}
\end{proposition}
\begin{proof}
If $\ell=r$, the claim immediately follows from Corollary~\ref{corr:Hrr(n)}, because
\begin{multline*}
\vert \{(i_1,\ldots,i_r)\in\N^r:P_r(t_1 i_1,\ldots,t_r i_r)\leq n\}\vert \\
=\vert \{(i_1,\ldots,i_r)\in\N^r: t_1 i_1\cdots t_r i_r\leq n\}\vert =\left\vert H_{r,r}\left(\left\lfloor \frac{n}{t_1 t_2\cdots t_r}\right\rfloor\right)\right\vert .
\end{multline*}
From now on, we assume that $\ell<r$. Write
\begin{align*}
\vert \{(i_1,\ldots,i_r)\in\N^r:P_\ell(t_1 i_1,\ldots,t_r i_r)\leq n\}\vert &=
\sum_{i_1=1}^{\infty}\cdots \sum_{i_r=1}^{\infty}\1_{\{P_\ell(t_1 i_1,\ldots,t_r i_r)\leq n\}}\\
&=\sum_{i_1=1}^{\infty}\cdots \sum_{i_r=1}^{\infty}\1_{\{P_\ell(t_1 i_1/n^{1/\ell},\ldots,t_r i_r/n^{1/\ell})\leq 1\}}.
\end{align*}
Applying Proposition~\ref{prop:dri} with the function $g(y_1,\ldots,y_r):=\1_{\{P_\ell(t_1 y_1,\ldots,t_r y_r)\leq 1\}}$ and using Proposition~\ref{prop:Hrl(n)}, we infer
\begin{multline*}
\lim_{n\to\infty}\frac{\vert \{(i_1,\ldots,i_r)\in\N^r:P_\ell(t_1 i_1,\ldots,t_r i_r)\leq n\}\vert }{\vert H_{\ell,r}(n)\vert }\\
=\frac{1}{\mathcal{V}_{\ell,r}}\int_0^\infty\cdots\int_0^\infty \1_{\{P_{\ell}(t_1y_1,\ldots,t_ry_r)\leq 1\}}{\rm d}y_1\cdots{\rm d}y_r=\left(\prod_{k=1}^{r}t_k\right)^{-1}.
\end{multline*}
For future use, we note here that
\begin{equation}\label{eq:prop1_proof0}
   \vert \{(i_1,\ldots,i_r)\in\N^r:P_\ell(t_1 i_1,\ldots,t_r i_r)\leq n\}\vert \leq\frac{n^{r/\ell}\mathcal{V}_{\ell,r}}{t_1t_2\cdots t_r},\quad 1\leq \ell<r,
\end{equation}
which is a direct consequence of monotonicity.
\end{proof}

\subsection{Proofs of Propositions~\ref{prop:joint_limit_l<r}, \ref{prop:joint_limit_l=r} and \ref{prop:product}}
\begin{proof}[Proof of Proposition~\ref{prop:joint_limit_l<r}]
The proof again relies on Proposition~\ref{prop:dri} from the Appendix. Note that
\begin{align*}
\mathbb{P}&\left\{V_1^{(n)}\leq \alpha_1 n^{1/\ell},\ldots,V_r^{(n)}\leq \alpha_r n^{1/\ell}\right\}\\
&=\frac{\vert \{(i_1,\ldots,i_r)\in\N^r: P_\ell(i_1,\ldots,i_r)\leq n,i_1\leq \alpha_1 n^{1/\ell},\ldots,i_r\leq \alpha_r n^{1/\ell} \}\vert }{\vert H_{\ell,r}(n)\vert }\\
&=\frac{n^{-r/\ell}\vert \{(i_1,\ldots,i_r)\in\N^r: P_\ell(i_1/n^{1/\ell},\ldots,i_r/n^{1/\ell})\leq 1,i_1/n^{1/\ell}\leq \alpha_1 ,\ldots,i_r/n^{1/\ell}\leq \alpha_r \}\vert }{n^{-r/\ell}\vert H_{\ell,r}(n)\vert},
\end{align*}
and the right-hand side converges, as $n\to\infty$, to
\begin{equation*}
   \frac{1}{\mathcal{V}_{\ell,r}}\int_0^{\alpha_1}\cdots\int_0^{\alpha_r}\1_{\{P_{\ell}(y_1,y_2,\ldots,y_r)\leq 1\}}{\rm d}y_1\cdots{\rm d}y_r.\qedhere
\end{equation*}
\end{proof}

We first prove Proposition~\ref{prop:product}, for this result will be used in the proof of Proposition~\ref{prop:joint_limit_l=r}.

\begin{proof}[Proof of Proposition~\ref{prop:product}]
For a
proof of \eqref{eq:product_conv1}, note that, for
$x\in [0,\,1]$ and $n\in\mathbb{N}$,
\begin{equation}\label{eq:1}
   \mathbb{P}\{V_{1}^{(n)}V_{2}^{(n)}\cdots V_{r}^{(n)}\leq xn\}=\frac{\vert \{(i_1,\ldots,i_r)\in\N^r: i_1i_2\cdots i_r\leq xn \}\vert }{\vert \{(i_1,\ldots,i_r)\in\N^r: i_1i_2\cdots i_r\leq n \}\vert },
\end{equation}
which, in view of Corollary~\ref{corr:Hrr(n)}, converges to $x$, as $n\to\infty$. As for \eqref{eq:product_conv2}, write
\begin{multline*}
\mathbb{P}\{V_{1}^{(n)}V_{2}^{(n)}\cdots V_{r}^{(n)} \leq x n^{r/\ell}\}=\frac{\vert \{(i_1,\ldots,i_r)\in\N^r: P_\ell(i_1,\ldots,i_r)\leq n,i_1\cdots i_r\leq xn^{r/\ell} \}\vert }{\vert H_{\ell,r}(n)\vert }\\
=\frac{n^{-r/\ell}\vert \{(i_1,\ldots,i_r)\in\N^r: P_\ell(i_1/n^{1/\ell},\ldots,i_r/n^{1/\ell})\leq 1, (i_1/n^{1/\ell})\cdots (i_r/n^{1/\ell})\leq x\}\vert }{n^{-r/\ell}\vert H_{\ell,r}(n)\vert }.
\end{multline*}
While the numerator converges to the integral on the right-hand side of \eqref{eq:product_distribution}, by Proposition~\ref{prop:dri} applied with $g(y_1,\ldots,y_r)=\1_{\{P_\ell(y_1,\ldots,y_r)\leq 1,\;y_1\cdots y_r\leq x\}}$, 
the denominator converges to $\mathcal{V}_{\ell,r}$, by Proposition~\ref{prop:Hrl(n)}. The value $x^{\ast}_{\ell,r}$ in \eqref{eq:product_distribution} 
is the supremum of the support of $U_{\ell,r}$. It can be found as the largest real number such that the surfaces $P_\ell(x_1,\ldots,x_r)=1$ and $x_1\cdots x_r=x^{\ast}_{\ell,r}$ have a nonempty intersection.

Formula \eqref{eq:product_bounded} is obvious for $\ell=r$, since, by construction, $(V_1^{(n)},\ldots,V_r^{(n)})$ is a point chosen at random 
in the set $H_{r,r}(n)$. Alternatively, \eqref{eq:product_bounded} follows on putting $x=1$ in \eqref{eq:1}. If $\ell<r$, formula \eqref{eq:product_bounded} can be proved as follows.
By definition, 
$P_\ell(V_1^{(n)},\ldots,V_r^{(n)})\leq n$, which implies
\begin{equation*}
\mathbb{P}\left\{\frac{V_{i_1}^{(n)}}{n^{1/\ell}}\frac{V_{i_2}^{(n)}}{n^{1/\ell}}\cdots \frac{V_{i_\ell}^{(n)}}{n^{1/\ell}}\leq 1\right\}=1,
\end{equation*}
for all $\ell$-tuples taken from $\{1,2,\ldots,r\}$. Multiplying all these inequalities, we obtain \eqref{eq:product_bounded}.
\end{proof}

\begin{proof}[Proof of Proposition~\ref{prop:joint_limit_l=r}]
We first observe that~\eqref{eq:product_conv1} implies
\begin{equation}\label{eq:proof_revision1}
\lim_{n\to\infty}\mathbb{P}\left\{V_1^{(n)}\cdots V_{r}^{(n)}\leq n^{\beta}\right\}=0,
\end{equation}
for every fixed $\beta<1$.

We shall prove a relation equivalent to \eqref{eq:joint_limit_l=r_alternative}, namely, for all $0<\beta_1<\ldots<\beta_{r-1}<\beta_{r}<1$ and sufficiently small $h_1,\ldots,h_{r-1},h_r>0$ such that the intervals
\begin{equation*}
   (\beta_1,\beta_1+h_1], (\beta_2,\beta_2+h_2],\ldots, (\beta_{r-1},\beta_{r-1}+h_{r-1}], (\beta_{r},\beta_{r}+h_{r}]
\end{equation*}
are disjoint,
\begin{multline}\label{eq:joint_limit_l=r_proof00}
\lim_{n\to\infty}\mathbb{P}\{V_1^{(n)}\in (n^{\beta_1},n^{\beta_1+h_1}],
V_1^{(n)}V_2^{(n)}\in (n^{\beta_2}, n^{\beta_2+h_2}],\ldots,V_1^{(n)}V_2^{(n)}\cdots V_{r}^{(n)}\in 
(n^{\beta_{r}},n^{\beta_{r}+h_{r}}]\}\\
=\mathbb{P}\{Z^{(1)}\in (\beta_1,\beta_1+h_1],Z^{(2)}\in (\beta_2,\beta_2+h_2],\ldots, Z^{(r)}\in (\beta_{r}, \beta_{r}+h_{r}]\}\\
=(r-1)! h_1\cdots h_{r-1}\1_{\{\beta_{r}+h_{r}\geq 1\}}.\hspace{75.6mm}
\end{multline}
The second equality in \eqref{eq:joint_limit_l=r_proof00} follows from the fact that $(Z^{(1)},\ldots,Z^{(r-1)})$ has a constant density in the region $\{(x_1,\ldots,x_{r-1})\in [0,1]^{r-1}:x_1\leq \cdots \leq x_{r-1}\leq 1\}$, which is equal to $(r-1)!$, see, for instance, formula (1.4) on p.~238 in \cite{Karlin:1968}. An appeal to~\eqref{eq:proof_revision1} and the fact that $V_1^{(n)}V_2^{(n)}\cdots V_{r}^{(n)}\leq n$ justifies the equivalence of~\eqref{eq:joint_limit_l=r_proof00} and
\begin{multline}\label{eq:joint_limit_l=r_proof0}
\lim_{n\to\infty}\mathbb{P}\{V_1^{(n)}\in (n^{\beta_1},n^{\beta_1+h_1}],
V_1^{(n)}V_2^{(n)}\in (n^{\beta_2}, n^{\beta_2+h_2}],\ldots,V_1^{(n)}V_2^{(n)}\cdots V_{r-1}^{(n)}\in (n^{\beta_{r-1}},n^{\beta_{r-1}+h_{r-1}}]\}\\
=(r-1)! h_1\cdots h_{r-1}.
\end{multline}

The probability on the left-hand side of \eqref{eq:joint_limit_l=r_proof0} is equal to
\begin{equation*}
   \frac{\vert \{(i_1,\ldots,i_r):i_1\in (n^{\beta_1},n^{\beta_1+h_1}], 
 \ldots, i_1\cdots i_{r-1}\in 
   (n^{\beta_{r-1}},n^{\beta_{r-1}+h_{r-1}}], i_1\cdots i_r\leq n\}\vert }{\vert H_{r,r}(n)\vert }.
\end{equation*}
Hence, 
according to Proposition~\ref{prop:Hrr(n)}, formula \eqref{eq:joint_limit_l=r_proof0} follows once we can check that the numerator 
is asymptotically equivalent to $h_1\cdots h_{r-1}n\log^{r-1}n$, as $n\to\infty$. The latter relation can be written as
\begin{multline*}
\sum_{i_1=1}^{\infty}\1_{\{i_1\in (n^{\beta_1}, n^{\beta_1+h_1}]\}}\sum_{i_2=1}^{\infty}\1_{\{i_2\in (n^{\beta_2}/i_1, n^{\beta_2+h_2}/i_1]\}}\cdots\sum_{i_{r-1}=1}^{\infty}\1_{\{i_{r-1}\in (n^{\beta_{r-1}}/(i_1\cdots i_{r-2}),n^{\beta_{r-1}+h_{r-1}}/(i_1\cdots i_{r-2})]\}}\\
\sum_{i_r=1}^{\infty}\1_{\{i_r \leq n/(i_1\cdot \ldots \cdot i_{r-1})\}}~\sim~h_1\cdots h_{r-1}n\log^{r-1}n,
\end{multline*}
or after calculating the rightmost sum as
\begin{multline}\label{eq:eq:joint_limit_l=r_proof1}
\sum_{i_1=1}^{\infty}\frac{\1_{\{i_1\in (n^{\beta_1},n^{\beta_1+h_1}]\}}}{i_1}\sum_{i_2=1}^{\infty}\frac{\1_{\{i_2\in (n^{\beta_2}/i_1,n^{\beta_2+h_2}/i_1]\}}}{i_2}\cdots\sum_{i_{r-1}=1}^{\infty}\frac{\1_{\{i_{r-1}\in (n^{\beta_{r-1}}/(i_1\cdots i_{r-2}),n^{\beta_{r-1}+h_{r-1}}/(i_1\cdots i_{r-2})]\}}}{i_{r-1}}\\
~\sim~h_1\cdots h_{r-1}\log^{r-1}n.
\end{multline}
Relation \eqref{eq:eq:joint_limit_l=r_proof1} readily follows by induction on $r\geq 2$ with the help of the formula 
\begin{equation*}
   \sum_{i=1}^{\infty}\frac{\1_{\{i\in [xn^{a},xn^{a+h}]\}}}{i}=h\log n+O(1),\quad n\to\infty,
\end{equation*}
which holds for all fixed $a,h>0$, uniformly in $x$ and $n$ satisfying 
$xn^{a}\to\infty$. In our setting, the latter relation is secured by 
$n^{\beta_{k-1}}/(i_1\cdots i_{k-2})\to \infty$ for every $k\geq 3$, which, in its turn, follows in view of $\beta_{k-2}+h_{k-2}<\beta_{k-1}$.
\end{proof}

\subsection{Prime decomposition}

The following proposition lies in the core of our main theorems and shows that as far as divisibility properties are concerned, the random vector $(V_1^{(n)},V_2^{(n)},\ldots,V_r^{(n)})$, uniformly distributed in the hyperbolic region $H_{\ell,r}(n)$, behaves as a set of $r$ independent variables uniformly distributed in $\{1,2,\ldots,n\}$, see, for example, Lemma~3.1 in \cite{Bostan+Marynych+Raschel:2019}.

\begin{proposition}\label{prop:convergence_to_geometrics}
Assume that $r\geq 2$. The following convergence in distribution holds true:
\begin{equation*}
   \left(\frac{V_{1}^{(n)}V_{2}^{(n)}\cdots V_{r}^{(n)}}{n^{r/\ell}},\left(\lambda_p(V_1^{(n)}),\ldots,\lambda_p(V_r^{(n)})\right)_{p\in\mathcal{P}}\right)\dodn \left(U_{\ell,r},\left(\mathcal{G}_1(p),\ldots,\mathcal{G}_r(p)\right)_{p\in\mathcal{P}}\right)
\end{equation*}
on $\mathbb{R}\times(\mathbb{R}^r)^{\infty}$, where $U_{\ell,r}$ on the right-hand side is independent of the $\mathcal{G}_k(p)$, for all $k=1,\ldots,r$ and $p\in\mathcal{P}$.
\end{proposition}
\begin{proof}
Fix $m\in\mathbb{N}$, $x\geq 0$, pairwise distinct primes $p_1,\ldots,p_m\in\mathcal{P}$ and arbitrary  $j_{k,t}\in\{0,1,2,\ldots\}$ for $k=1,\ldots,r$ and $t=1,\ldots,m$. Write
\begin{align*}
&\hspace{-0.8cm}\mathbb{P}\{V_{1}^{(n)}V_{2}^{(n)}\cdots V_{r}^{(n)}\leq x n^{r/\ell},\lambda_{p_t}(V_k^{(n)})\geq j_{k,t}\text{ for all }k=1,\ldots,r\text{ and }t=1,\ldots,m\}\\
&=\mathbb{P}\{V_{1}^{(n)}V_{2}^{(n)}\cdots V_{r}^{(n)}\leq x n^{r/\ell},p_t^{j_{k,t}}\text{ divides }V_k^{(n)}\text{ for all }k=1,\ldots,r\text{ and }t=1,\ldots,m\}\\
&=\frac{1}{\vert H_{\ell,r}(n)\vert }\sum_{i_1=1}^{\infty}\cdots\sum_{i_r=1}^{\infty}\1\{P_\ell(i_1,\ldots,i_r)\leq n,i_1\cdots i_r\leq xn^{r/\ell},\\
&\hspace{4.5cm}p_t^{j_{k,t}}\text{ divides } i_k \text{ for all }k=1,\ldots,r\text{ and }t=1,\ldots,m\}\\
&=\frac{1}{\vert H_{\ell,r}(n)\vert }\sum_{i_1=1}^{\infty}\cdots\sum_{i_r=1}^{\infty}\1\Big\{P_\ell(i_1,\ldots,i_r)\leq n,i_1\cdots i_r\leq xn^{r/\ell},\\
&\hspace{4.5cm}\prod_{t=1}^{m} p_t^{j_{k,t}}\text{ divides } i_k \text{ for all }k=1,\ldots,r\Big\}.
\end{align*}
For notational simplicity, 
put $\mu_k:=\prod_{t=1}^{m} p_t^{j_{k,t}}$. Since the sum over $i_k$ in the formula above is 
actually
taken over multiples of $\mu_k$, $k=1,\ldots,r$, we obtain
\begin{multline}\label{eq:lambda_dis}
\mathbb{P}\{V_{1}^{(n)}V_{2}^{(n)}\cdots V_{r}^{(n)}\leq x n^{\ell/r},\lambda_{p_t}(V_k^{(n)})\geq j_{k,t}\text{ for all }k=1,\ldots,r\text{ and }t=1,\ldots,m\}\\
=\frac{\vert \{(i_1,\ldots,i_r)\in\N^r:P_\ell(\mu_1 i_1,\ldots,\mu_r i_r)\leq n,(\mu_1i_1)\cdots(\mu_r i_r)\leq xn^{r/\ell}\}\vert }{\vert H_{\ell,r}(n)\vert }.
\end{multline}
If $\ell=r$, the last quantity converges as $n\to\infty$ to $x/(\mu_1\cdots\mu_r)$, by Corollary~\ref{corr:Hrr(n)}. If $\ell<r$, it converges to
\begin{equation*}
   \frac{1}{\mathcal{V}_{\ell,r}}\int_0^\infty\cdots\int_0^\infty \1_{\{P_\ell(\mu_1 y_1,\ldots,\mu_r y_r)\leq 1,\; (\mu_1 y_1)\cdots (\mu_r y_r)\leq x\}}{\rm d}y_1\cdots{\rm d}y_r
=\frac{1}{\mu_1\cdots\mu_r}\mathbb{P}\{U_{\ell,r}\leq x\},
\end{equation*}
by Proposition~\ref{prop:dri}. This finishes the proof, because
\begin{equation*}
   \frac{1}{\mu_1\cdots\mu_r}=\mathbb{P}\{\mathcal{G}_k(p_t) \geq j_{k,t}\text{ for all }k=1,\ldots,r\text{ and }t=1,\ldots,m\}.\qedhere
\end{equation*}
\end{proof}

\subsection{Proof of Theorem~\ref{thm:main1}}
We start by noting that 
the infinite product on the right-hand side of \eqref{thm:main1_claim} converges almost surely (a.s.) and in mean. For $r=2$, a
proof can 
be found in formula (6.8) of \cite{AlsKabMar:2019}, 
see also \cite{Diaconis+Erdos:2004}. Since the infinite product is nonincreasing in $r$ a.s., it must also converge for all $r\geq 3$.

We shall use a representation 
\begin{equation*}
   \GCD(V_1^{(n)},V_2^{(n)},\ldots,V_r^{(n)}))=\prod_{p\in\mathcal{P}}p^{\min_{k=1,\ldots,r}\lambda_p(V_k^{(n)})}=\left(\prod_{p\in\mathcal{P},p\leq M}\cdots\right)\left(\prod_{p\in\mathcal{P},p > M}\cdots\right),
\end{equation*}
where $M>0$ is a fixed large number. As $n\to\infty$, the first product converges in distribution to
\begin{equation*}
   \prod_{p\in\mathcal{P},p\leq M}p^{\min_{k=1,\ldots,r}\mathcal{G}_k(p)},
\end{equation*}
which, in its turn,  
is a.s.\ converging, as $M\to\infty$, to the right-hand side of \eqref{thm:main1_claim}. According to Theorem~3.2 in \cite{Billingsley:1999}, it remains to check that
\begin{equation*}
\lim_{M\to\infty}\limsup_{n\to\infty}\mathbb{P}\left\{\prod_{p\in\mathcal{P},p > M}p^{\min_{k=1,\ldots,r}\lambda_p(V_k^{(n)})}\neq 1\right\}=0,
\end{equation*}
which is equivalent to
\begin{equation}\label{eq:bill12}
\lim_{M\to\infty}\limsup_{n\to\infty}\mathbb{P}\left\{\text{ for some }p\in\mathcal{P},p>M, \min_{k=1,\ldots,r}\lambda_p(V_k^{(n)})>0\right\}=0.
\end{equation}
Using Boole's inequality and formula \eqref{eq:lambda_dis} we write
\begin{align*}
&\hspace{-2cm}\mathbb{P}\left\{\text{ for some }p\in\mathcal{P},p>M, \min_{k=1,\ldots,r}\lambda_p(V_k^{(n)})>0\right\}\\
&\leq\sum_{p\in\mathcal{P},p>M}\mathbb{P}\{\lambda_p(V_1^{(n)})\geq 1,\ldots,\lambda_p(V_r^{(n)})\geq 1\}\\
&=\sum_{p\in\mathcal{P},p>M}\frac{\vert \{(i_1,\ldots,i_r)\in\N^r: P_\ell(p i_1,\ldots, p i_r)\leq n\}\vert }{\vert H_{\ell,r}(n)\vert }\\
&=\sum_{p\in\mathcal{P},p>M}\frac{\vert \{(i_1,\ldots,i_r)\in\N^r: P_\ell(i_1,\ldots, i_r)\leq n/p^\ell\}\vert }{\vert H_{\ell,r}(n)\vert }\\
&=\sum_{p\in\mathcal{P},p>M}\frac{\vert H_{\ell,r}(\lfloor n/p^\ell\rfloor)\vert }{\vert H_{\ell,r}(n)\vert }.
\end{align*}
Invoking 
Corollaries~\ref{corr:Hrr(n)} and \ref{cor:counting_function_reg_var} in conjunction with Potter's bound for regularly varying functions 
(Theorem~1.5.6 in \cite{BGT}), we infer, 
for $n\in\N$ large enough,
\begin{equation*}
   \frac{\vert H_{\ell,r}(\lfloor n/p^\ell\rfloor)\vert }{\vert H_{\ell,r}(n)\vert}\leq \frac{2}{(p^{\ell})^{(2r-1)/(2\ell)}}=\frac{2}{p^{r-1/2}}\leq \frac{2}{p^{3/2}}.
\end{equation*}
This yields \eqref{eq:bill12}, because
\begin{equation*}
   \lim_{M\to\infty}\sum_{p\in\mathcal{P},p>M}\frac{2}{p^{3/2}}=0.
\end{equation*}

\subsection{Proof of Theorem~\ref{thm:main2} and Corollary~\ref{corr:LCM_moments}}
Similarly to the proof of Theorem~\ref{thm:main1}, we start with a
decomposition
\begin{multline*}
\frac{\LCM(V_1^{(n)},V_2^{(n)},\ldots,V_r^{(n)})}{V_1^{(n)}V_2^{(n)}\cdots V_r^{(n)}}=\prod_{p\in\mathcal{P}}p^{\max_{k=1,\ldots,r}\lambda_p(V_k^{(n)})-\sum_{k=1}^{r}\lambda_p(V_k^{(n)})}\\
=\left(\prod_{p\in\mathcal{P},p\leq M}\cdots\right)\left(\prod_{p\in\mathcal{P},p > M}\cdots\right),
\end{multline*}
where $M$ is a fixed large integer. As $n\to\infty$, the first product converges to
\begin{equation*}
   \prod_{p\in\mathcal{P},p\leq M}p^{\max_{k=1,\ldots,r}\mathcal{G}_k(p)-\sum_{k=1}^{r}\mathcal{G}_k(p)}
\end{equation*}
by virtue of Proposition~\ref{prop:convergence_to_geometrics}. As $M\to\infty$, the latter converges a.s.~to
\begin{equation*}
   \prod_{p\in\mathcal{P}}p^{\max_{k=1,\ldots,r}\mathcal{G}_k(p)-\sum_{k=1}^{r}\mathcal{G}_k(p)},
\end{equation*}
which is an a.s.~finite random variable, see Proposition 2.1 in \cite{Bostan+Marynych+Raschel:2019}.

Appealing once again to Theorem~3.2 in \cite{Billingsley:1999}, we see that it is enough to check that
\begin{equation*}
\lim_{M\to\infty}\limsup_{n\to\infty}\mathbb{P}\left\{\prod_{p\in\mathcal{P},p > M}p^{\max_{k=1,\ldots,r}\lambda_p(V_k^{(n)})-\sum_{k=1}^{r}\lambda_p(V_k^{(n)})}\neq 1\right\}=0,
\end{equation*}
which is equivalent to
\begin{equation}\label{eq:bill123}
\lim_{M\to\infty}\limsup_{n\to\infty}\mathbb{P}\left\{\text{ for some }p\in\mathcal{P},p>M, \max_{k=1,\ldots,r}\lambda_p(V_k^{(n)})\neq \sum_{k=1}^{r}\lambda_p(V_k^{(n)})\right\}=0.
\end{equation}
Observe that
\begin{equation*}
   \left\{\max_{k=1,\ldots,r}\lambda_p(V_k^{(n)})\neq \sum_{k=1}^{r}\lambda_p(V_k^{(n)})\right\}\subset
\left\{\sum_{k=1}^{r}\lambda_p(V_k^{(n)})\geq 2\right\}.
\end{equation*}
Thus
\begin{align*}
\mathbb{P}&\left\{\text{ for some }p\in\mathcal{P},p>M, \max_{k=1,\ldots,r}\lambda_p(V_k^{(n)})\neq \sum_{k=1}^{r}\lambda_p(V_k^{(n)})\right\}\\
&\leq \sum_{p\in\mathcal{P},p>M}\mathbb{P}\left\{\sum_{k=1}^{r}\lambda_p(V_k^{(n)})\geq 2\right\}\\
&\leq \sum_{p\in\mathcal{P},p>M}\mathbb{P}\left\{\lambda_p(V_k^{(n)})\geq 2\text{ for some }k=1,\ldots,r\right\}\\
&\hspace{1cm}+\sum_{p\in\mathcal{P},p>M}\mathbb{P}\left\{\lambda_p(V_i^{(n)})\geq 1,\lambda_p(V_j^{(n)})\geq 1 \text{ for some }i,j=1,\ldots,r,\,i\neq j\right\}.
\end{align*}
Using the fact that the vector $(V_1^{(n)},V_2^{(n)},\ldots,V_r^{(n)})$ is exchangeable, that is, its distribution is invariant under permutations, and then applying 
formula~\eqref{eq:lambda_dis}, we conclude 
that
\begin{align*}
\mathbb{P}&\left\{\text{ for some }p\in\mathcal{P},p>M, \max_{k=1,\ldots,r}\lambda_p(V_k^{(n)})\neq \sum_{k=1}^{r}\lambda_p(V_k^{(n)})\right\}\\
&\leq r\sum_{p\in\mathcal{P},p>M}\mathbb{P}\left\{\lambda_p(V_1^{(n)})\geq 2\right\}+r(r-1)\sum_{p\in\mathcal{P},p>M}\mathbb{P}\left\{\lambda_p(V_1^{(n)})\geq 1,\lambda_p(V_2^{(n)})\geq 1\right\}\\
&=r\sum_{p\in\mathcal{P},p>M}\frac{\vert \{(i_1,\ldots,i_r)\in\N^r:P_\ell(p^2 i_1,i_2,\ldots,i_r)\leq n\}\vert }{\vert H_{\ell,r}(n)\vert }\\
&\hspace{1cm}+r(r-1)\sum_{p\in\mathcal{P},p>M}\frac{\vert \{(i_1,\ldots,i_r)\in\N^r:P_\ell(p i_1,p i_2,i_3,\ldots,i_r)\leq n\}\vert }{\vert H_{\ell,r}(n)\vert }.
\end{align*}
If $\ell=r$, 
the right-hand side is equal to
\begin{equation*}
   r^2\sum_{p\in\mathcal{P},p>M}\frac{\vert H_{r,r}(\lfloor n/p^2\rfloor)\vert }{\vert H_{r,r}(n)\vert }
\end{equation*}
and \eqref{eq:bill123} follows by appealing to 
Potter's bound in the same fashion as we did in the proof of Theorem~\ref{thm:main1}. If $\ell<r$, we 
apply inequality~\eqref{eq:prop1_proof0} to obtain
\begin{equation*}
   \mathbb{P}\left\{\text{ for some }p\in\mathcal{P},p>M, \max_{k=1,\ldots,r}\lambda_p(V_k^{(n)})\neq \sum_{k=1}^{r}\lambda_p(V_k^{(n)})\right\}
\leq r^2 \frac{n^{r/\ell}\mathcal{V}_{\ell,r}}{\vert H_{\ell,r}(n)\vert }\left(\sum_{p\in\mathcal{P},p>M}\frac{1}{p^2}\right).
\end{equation*}
Sending first $n\to\infty$ and using Proposition~\ref{prop:Hrl(n)}, and then 
letting $M\to\infty$ yields~\eqref{eq:bill123}. Thus, \eqref{eq:lcm_conv1} has been 
proved. The second limit relation  \eqref{eq:lcm_conv2} is justified 
by the continuous mapping theorem in combination with the joint convergence 
\begin{equation*}
   \left(\frac{V_1^{(n)}V_2^{(n)}\cdots V_r^{(n)}}{n^{r/\ell}},\frac{\LCM(V_1^{(n)},V_2^{(n)},\ldots,V_r^{(n)})}{V_1^{(n)}V_2^{(n)}\cdots V_r^{(n)}}\right)~\dodn~\left(U_{\ell,r},\prod_{p\in\mathcal{P}}
p^{\max_{k=1,\ldots,r}\mathcal{G}_k(p)-\sum_{k=1}^{r}\mathcal{G}_k(p)}\right),
\end{equation*}
which holds true, by Proposition~\ref{prop:convergence_to_geometrics}. The convergence of all power moments of positive orders follows from the fact that both variables on the left-hand side are supported by $[0,1]$.

Corollary~\ref{corr:LCM_moments} follows immediately from formula \eqref{eq:lcm_conv2} and Proposition~\ref{prop:reg_var_moments} in the Appendix, upon applying the Skorohod representation theorem, see, for instance, Theorem~4.30 in \cite{Kallenberg:2001}. The theorem 
guarantees that there exist 
versions of the 
random variables on the left-hand side of \eqref{eq:lcm_conv2}, 
which converge almost surely to a version 
of the limit random variable in \eqref{eq:lcm_conv2}. 

\appendix
\section{Two convergence results}
\label{sec:app}

First, we state a result concerning 
multivariate infinite Riemann sums. 
\begin{proposition}\label{prop:dri}
Let $r\in\mathbb{N}$ and $g:\mathbb{R}^r_{\geq 0}\to \mathbb{R}_{\geq 0}$ be a coordinatewise nonincreasing function. Assume that
\begin{equation*}
   I:=\int_0^\infty\cdots\int_0^\infty g(y_1,y_2,\ldots,y_r){\rm d}y_{1}\cdots{\rm d}y_{r}<\infty.
\end{equation*}
Then
\begin{equation*}
   \lim_{n\to\infty }\frac{1}{n^{r}}\sum_{i_1=1}^{\infty}\cdots\sum_{i_r=1}^{\infty}g\left(\frac{i_1}{n},\frac{i_2}{n},\ldots,\frac{i_r}{n}\right)=I.
\end{equation*}
\end{proposition}
\begin{proof}
Put
\begin{equation*}
   I_n:=\int_{1/n}^\infty\cdots\int_{1/n}^\infty g(y_1,y_2,\ldots,y_r){\rm d}y_{1}\cdots{\rm d}y_{r}
\end{equation*}
and note that, by monotonicity,
\begin{equation*}
   I_n\leq \frac{1}{n^{r}}\sum_{i_1=1}^{\infty}\cdots\sum_{i_r=1}^{\infty}g\left(\frac{i_1}{n},\frac{i_2}{n},\ldots,\frac{i_r}{n}\right)\leq I.
\end{equation*}
By the dominated convergence theorem,
\begin{equation*}
   0\leq I-I_n = \int_{0}^\infty\cdots\int_{0}^\infty g(y_1,y_2,\ldots,y_r)\1_{\{\min(y_1,y_2,\ldots,y_r)\leq n^{-1}\}}{\rm d}y_{1}\cdots{\rm d}y_{r}\to 0,\quad n\to\infty.\qedhere
\end{equation*}
\end{proof}

Proposition~\ref{prop:reg_var_moments} 
is used in the proof of the moment convergence 
in Theorem~\ref{corr:LCM_moments}. Even though the result 
looks rather standard, 
we have not been able to locate it in the literature.
\begin{proposition}\label{prop:reg_var_moments}
Assume that $X$ is a random variable with $\mathbb{P}\{X=0\}<1$, and $(X_n)_{n\in\N}$ is a sequence of random variables on some probability space $(\Omega,\mathcal{F},\mathbb{P})$ such that, $\mathbb{P}$-a.s.,
\begin{equation*}
   \frac{X_n}{a_n}\to X
\quad\text{ as \quad }n\to\infty,\quad\text{and}\quad 0\leq \frac{X_n}{a_n}\leq C\quad \text{for some 
constant } C>0,
\end{equation*}
where 
$a_n\to\infty$. Let $f:\mathbb{R}_{\geq 0}\to\mathbb{R}$ be a locally bounded function which varies regularly at $\infty$ of index 
$\beta>0$. Then, as $n\to\infty$,
\begin{equation*}
   \E f(X_n)~\sim~(\E X^{\beta})f(a_n).
\end{equation*}
\end{proposition}
\begin{proof}
By Theorem~1.5.3 in \cite{BGT}, there exists a nondecreasing function $g$ such that $g(x)\sim f(x)$, as $x\to\infty$. Fix $\varepsilon>0$ and write
\begin{equation*}
   \frac{g(X_n)}{g(a_n)}=\frac{g((X_n/a_n)a_n)}{g(a_n)}=\frac{g((X_n/a_n)a_n)}{g(a_n)}\1_{\{X_n/a_n>\varepsilon\}}+\frac{g((X_n/a_n)a_n)}{g(a_n)}\1_{\{X_n/a_n\leq \varepsilon\}}=:I_n(\varepsilon)+J_n(\varepsilon).
\end{equation*}
By the uniform convergence theorem for regularly varying functions (Theorem~1.5.2 in \cite{BGT}),
\begin{equation*}
\lim_{n\to\infty}I_n(\varepsilon)=X^{\beta}\1_{\{X>\varepsilon\}}\quad \mathbb P-\text{a.s.}
\end{equation*}
By monotonicity,
\begin{equation*}
   \limsup_{n\to\infty}J_n(\varepsilon)\leq \varepsilon^{\beta}\quad \mathbb P-\text{a.s.}
\end{equation*}
and thereupon
\begin{equation*}
   \limsup_{n\to\infty}\frac{g(X_n)}{g(a_n)}\leq X
   ^{\beta}\1_{\{X
   >\varepsilon\}}+\varepsilon^{\beta}\quad \mathbb P-\text{a.s.}
\end{equation*}
Hence, 
\begin{equation*}
   \limsup_{n\to\infty}\frac{g(X_n)}{g(a_n)}\leq X^{\beta}\quad \mathbb P-\text{a.s.}
\end{equation*}
The converse inequality for the $\liminf$ is a consequence of
\begin{equation*}
   \frac{g(X_n)}{g(a_n)}\geq \frac{g(X_n)}{g(a_n)}\1_{\{X_n/a_n>\varepsilon\}}\to X^{\beta}\1_{\{X
   >\varepsilon\}},\quad n\to\infty\quad \mathbb P-\text{a.s.} 
\end{equation*}
Thus,
\begin{equation*}
\lim_{n\to\infty}\frac{g(X_n)}{g(a_n)}=X^\beta\quad \mathbb{P}-\text{a.s.}  
\end{equation*}
By monotonicity and regular variation of $g$ in conjunction with the assumption $X_n/a_n\leq C$, the left-hand side is bounded, which entails
\begin{equation*}
\lim_{n\to\infty}\frac{\E g(X_n)}{g(a_n)}= \E X^{\beta}. 
\end{equation*}
Further, 
\begin{equation*}
\lim_{n\to\infty}\frac{\E g(X_n)}{f(a_n)}= \E X^{\beta}.
\end{equation*}
It remains to note that
\begin{equation}\label{eq:reg_var_moment_proof1}
\lim_{n\to\infty}\frac{\E g(X_n)}{\E f(X_n)}=1.
\end{equation}
Indeed, given
$\varepsilon>0$, there exists $x_0>0$ such that $(1-\varepsilon)f(x)\leq g(x)\leq (1+\varepsilon)f(x)$, for all $x\geq x_0$. Thus,
\begin{equation*}
   (1-\varepsilon)\E f(X_n)-(1-\varepsilon)\sup_{x\in [0,\,x_0]}f(x)\leq (1-\varepsilon)\E f(X_n)\1_{\{X_n>x_0\}} \leq \E g(X_n)\leq (1+\varepsilon)\E f(X_n)+g(x_0),
\end{equation*}
and \eqref{eq:reg_var_moment_proof1} follows.
\end{proof}

\end{document}